\documentclass[10pt]{article}

\usepackage{mathptmx}
\usepackage{eucal}
\usepackage{amsmath}
\usepackage{amscd}
\usepackage{amssymb}
\usepackage{amsthm}
\usepackage{xspace}
\usepackage[all,tips]{xy}
\usepackage[dvips]{graphicx}
\usepackage{graphics}
\usepackage{epsfig}
\usepackage{verbatim}
\usepackage{syntonly}
\usepackage{hyperref}
\usepackage{amssymb, color, url}

\newtheorem{theorem}{Theorem}
\newtheorem{claim}{Claim}

\newtheorem{lemma}[theorem]{Lemma}
\newtheorem{question}{Question}

\newtheorem*{theoremNL}{Theorem}

\input xy
\xyoption{all}

\DeclareMathOperator{\Aut}{Aut}

\DeclareMathOperator{\cd}{cd}

\usepackage{fancyhdr}

\fancypagestyle{plain}

\begin{document}
\bibliographystyle{plain}


\title{\textbf{Parasurface groups}}
\author{K. Bou-Rabee}\maketitle


\begin{abstract}
A residually nilpotent group is \emph{$k$-parafree} if all of its lower central series quotients match those of a free group of rank $k$.
Magnus proved that $k$-parafree groups of rank $k$ are themselves free.
We mimic this theory with surface groups playing the role of free groups.
Our main result shows that the analog of Magnus' Theorem is false in this setting.
\end{abstract}

\section*{Introduction}

This article is motivated by three stories.
The first story concerns a theorem of Wilhelm Magnus.
Recall that the \emph{lower central series} of a group $G$ is defined to be
\[
 \gamma_1(G) := G \text{ and }\gamma_k(G) := [G, \gamma_{k-1}(G)] \text{ for $k \geq 2$},
 \]
 where $[A,B]$ denotes the group generated by commutators of elements of $A$ with elements of $B$. The \emph{rank} of $G$ is the size of a minimal generating set of $G$. In 1939, Magnus gave a beautiful characterization of free groups in terms of their lower central series \cite{M1}.

\begin{theoremNL}[Magnus' theorem on parafree groups] \label{magnusTheorem}
Let $F_k$ be a nonabelian free group of rank $k$ and $G$ a group of rank $k$.
If $G/\gamma_i(G) \cong F_k/\gamma_i(F_k)$ for all $i$, then $G \cong F_k$.
\end{theoremNL}

Following this result, Hanna Neumann inquired whether it was possible for two residually nilpotent groups $G$ and $G'$ to have $G/\gamma_i(G) \cong G'/\gamma_i(G')$ for all $i$ without having $G \cong G'$ (see \cite{L1}).
Gilbert Baumslag \cite{baumslag-1976} gave a positive answer to this question by constructing what are now known as parafree groups which are not themselves free. A group $G$ is \emph{parafree} if:

\begin{enumerate}
\item $G$ is residually nilpotent, and
\item there exists a finitely generated free group $F$ with the property that $G/\gamma_i(G) \cong F/\gamma_i(F)$ for all $i$.
\end{enumerate}

By Magnus' Theorem, Baumslag's examples necessarily have rank different from the corresponding free group.
In this paper we give new examples addressing Neumann's question.
Specifically, we construct residually nilpotent groups $G$ which share the same lower central series quotients with a surface group but are not surface groups themselves.
These examples are analogous to Baumslag's parafree groups, where the role of free groups is replaced by surface groups.
Consequently, we call such groups \emph{parasurface groups}.
As the parasurface examples constructed in this paper have the same rank as their corresponding surface groups, the analog of Magnus' Theorem for parasurface groups is false.

\begin{theorem} \label{magnusparasurfacethm}
Let $\Gamma_g$ be the genus $g$ surface group.
There exists a rank $2g$ residually nilpotent group $G$ such that $G/\gamma_i(G) \cong \Gamma_g / \gamma_i(\Gamma_g)$ for all $i$.
\end{theorem}

\vskip.1in
We now turn towards the second story, which concerns residual properties in free groups.
Our second story requires some notation.
We say two elements $x, y \in G$ are \emph{nilpotent--conjugacy--equivalent}, if the images of $g$ and $h$ in all nilpotent quotients of $G$ are conjugate.
A group $G$ is said to be \emph{conjugacy--nilpotent--separable} if any pair of nilpotent--conjugacy--equivalent elements must be conjugate.
Free groups are known to be conjugacy--nilpotent--separable (see \cite{lyndon-schupp} and \cite{PL}).
A natural question to ask is whether the role of inner automorphisms may be played by automorphisms.
In this vein, we say two elements $g, h \in G$ are \emph{automorphism--equivalent}, if there exists some $\phi \in \Aut G$ with $\phi(g) = h$.
Further, we say $g$ is \emph{nilpotent--automorphism--equivalent} to $h$ if the images of $g$ and $h$ in all nilpotent quotients of $G$ are automorphism--equivalent.
A group G is \emph{automorphism--nilpotent--separable}, if all pairs of nilpotent--automorphism--equivalent elements must be automorphism--equivalent.
In contrast to free groups being conjugacy--nilpotent--separable, we have

\begin{theorem} \label{AutFreeNilpotentSeparableTheorem}
Nonabelian free groups of even rank are not automorphism--nilpotent--separable.
\end{theorem}

\noindent
In the same flavor, Orin Chein in 1969 showed that there exist automorphisms of nilpotent quotients of $F_3$ that do not lift to automorphisms of $F_3$  \cite{chein-1969}.

\vskip.1in

Our third and final story concerns Magnus' conjugacy theorem for groups with one defining relator (Theorem 4.11, page 261, in \cite{MKS}), which we state below.
Let $\left<\left< w \right>\right>_H$ be the group generated by the $H$-conjugates of $w$ for a subgroup $H$ of $G$.

\begin{theoremNL}[Magnus' conjugacy theorem for groups with one defining relator]
\label{conjthm}
Let $s, t$ be elements in $F_k$ such that $\left< \left< s \right>\right>_{F_k} = \left< \left< t \right>\right>_{F_k}$.
Then $s$ is conjugate to $t^\epsilon$ for $\epsilon = 1$ or $-1$.
\end{theoremNL}

\noindent
Our next result demonstrates that this theorem does \emph{not} generalize to one-relator nilpotent groups.
Let $F_{k,i} = F_k/\gamma_i(F_k)$ be the \emph{free rank $k$, $i$-step nilpotent quotient}.

\begin{theorem}\label{conjugacytheoremNoNo}
Let $F_4 = \left< a, b, c, d\right>$ denote the free group of rank $4$ and $w = [a,b][c,d]$.
Then 
$$\left< \left< w[w,bwb^{-1}] \right> \right>_{F_{4,i}}= \left<\left< w \right>\right>_{F_{4,i}}$$
for all $i$. However, for large $i$, the element $w$ is not conjugate to $w[w,bwb^{-1}]$ in $F_{4,i}$.
\end{theorem}

\subsubsection*{Acknowledgements}

I am extremely grateful to my advisor, Benson Farb, for his teaching and inspiration.
I am also especially grateful to my coadvisor, Ben McReynolds, for all his time and help.
Further, I thank Justin Malestein for some useful discussions and careful reading of this paper.
Finally, I would like to thank Alex Wright, Tom Church, Blair Davey, and Thomas Zamojski  for reading early drafts of this paper.

\section{Almost surface groups} \label{pssection}

Let $\Gamma_g$ be the fundamental group of a closed hyperbolic surface of genus $g$.
A group $G$ is a \emph{weakly $g$-parasurface group} if $G/\gamma_k(G) \cong \Gamma_g / \gamma_k(\Gamma_g)$ for all $k \geq 1$.
If $G$ is weakly $g$-parasurface and residually nilpotent, we say that $G$ is $g$-\emph{parasurface}.
Let $G = F/N$ be a weakly $g$-parasurface group where $F$ is a free group of rank $2g$ on generators $a_1, a_2, \ldots, a_{2g-1}, a_{2g}$.
Set $w = [a_1, a_2] \cdots [a_{2g-1}, a_{2g}]$.
Recall that $\left<\left< w \right>\right>_F$ is the normal closure of $w$ in $F$.
Then we have the following trichotomy for such groups $G$:
\begin{itemize}
\item[] Type I. There exists an isomorphism $\phi: F \to F$ such that $\phi(N) \geq \left<\left< w \right>\right>_F$.
\item[] Type II. There exists an isomorphism $\phi: F \to F$ such that $\phi(N) <  \left<\left<w\right>\right>_F$.
\item[] Type III. $G$ is not of Type I or II.
\end{itemize}

The following theorem demonstrates that only examples of Type I or III may be residually nilpotent.

\begin{theorem} \label{typeTheorem}
Groups of Type I must be surface groups.
Further, groups of Type II are never parasurface.
\end{theorem}

Our next two theorems show that although surface groups are residually nilpotent, there exists examples of Type II and III.
That is, weakly parasurface groups which are not parasurface groups exist.
And, further, parasurface groups which are not surface groups exist.

\begin{theorem} \label{weaklyparasurfaceTheoremExamples}
Let $k > 2$ be even, let $F_k = \left<a_1,a_2, \ldots, a_k\right>$, let $w = [a_1,a_2]\cdots [a_{k-1}, a_k]$, and let $\gamma$ be an element in $F_k$.
If $w[w,\gamma w\gamma^{-1}]$ is cyclically reduced and of different word length than $w$ in $F_k$, then the group
$$G = \left< a_1,a_2, \ldots, a_k : w[w,\gamma w \gamma^{-1}] \right>$$
is weakly $k/2$-parasurface of Type II.
\end{theorem}

\noindent
In Theorem \ref{weaklyparasurfaceTheoremExamples}, one can take $\gamma = a_2$, for example.

\begin{theorem} \label{parasurfaceTheoremExamples}
Let $k > 2$ be even and let $F_k = \left<a_1, a_2, \ldots, a_k \right>$ and let $\delta$ be in the commutator subgroup of $F_k$.
If $[a_1 \delta,a_2]$ is cyclically reduced and is of different word length than $[a_1,a_2]$ in $F_2$, then the group
$$G = \left< a_1, a_2, \ldots, a_k : [a_1 \delta, a_2] [a_3, a_4]\cdots [a_{k-1}, a_k] \right>$$
is $k/2$-parasurface of Type III.
\end{theorem}

\noindent
In Theorem \ref{parasurfaceTheoremExamples} one can take $\delta = [[a_1,a_2],a_1]$, for example.

\section{Proofs of the main results} \label{proofSection}

\subsection{Preliminaries}

We first list a couple of results needed in the proofs of our main theorems.
For the following lemma see Lemma 5.9, page 350 in Magnus, Korass and Solitar \cite{MKS}.
\begin{lemma} \label{generatorlemma}
Let $G$ be a free nilpotent group of class $c$ and let $g_1, g_2,  \ldots \in G$ be elements whose projections to $G/[G, G]$ generate.
Then $g_1, g_2,  \ldots$ generate $G$. \qed
\end{lemma}

For the following theorem, see Azarov \cite{azarov}, Theorem 1.
\begin{theorem}[Azarov's Theorem] \label{azarovTheorem}
Let $A$ and $B$ be free groups, and let $\alpha$ and $\beta$ be nonidentity elements of the groups $A$ and $B$, respectively.
Let $G = (A*B ; \alpha = \beta)$.
Let $n$ be the largest positive integer such that $y^n = \beta$ has a solution in $B$.
If $n = 1$, then $G$ is a residually finite $p$-group.
\end{theorem}

\noindent

\subsection{The proofs}

Before proving Theorems  \ref{magnusparasurfacethm}, \ref{AutFreeNilpotentSeparableTheorem}, and \ref{conjugacytheoremNoNo} from the introduction, we first prove Theorems \ref{typeTheorem}, \ref{weaklyparasurfaceTheoremExamples}, and \ref{parasurfaceTheoremExamples} from Section \ref{pssection}.
\paragraph{Proof of Theorem \ref{typeTheorem}.}
We first show that groups of Type I must be surface groups.
For the sake of a contradiction, suppose that $G$ is a weakly $g$-parasurface of Type I, and $G$ is not isomorphic to $\Gamma_g$.
Let $F$ and $w$ be as in the definition of Type I groups.
By assumption, there exists an isomorphism $\phi: F \to F$ such that $\phi(N) \geq \left<\left< w \right>\right>_F$.
The isomorphism $\phi^{-1}$ induces a homomorphism $\rho_i : \Gamma_g/\gamma_i(\Gamma_g) \to G/ \gamma_i(G)$ which is surjective for all $i$.
As finitely generated nilpotent groups are Hopfian (see Section III.A.19 in \cite{harpe-2000}, for instance), the maps $\rho_i$ must be isomorphisms for all $i$.
On the other hand, since $G$ is not isomorphic to $\Gamma_g$, we must have some $\gamma \in \phi(N) - \left< \left< w \right> \right>_F$.
Further, $F/\left< \left< w \right> \right>_F = \Gamma_g$ is residually nilpotent, so there exists some $i$ such that $\gamma \neq 1$ in $\Gamma_g/ \gamma_i(\Gamma_g)$.
Since $\gamma \in \ker \rho_i$, we have a contradiction.

We now show that groups of Type II are never residually nilpotent.
For the sake of a contradiction, suppose that $G$ is a residually nilpotent group of Type II which is not a surface group.
Let $F$ and $w$ be as in the definition of Type II groups.
By assumption, the map $\phi: F \to F$, induces a map $\psi: G \to \pi_1(\Gamma_g)$ that is onto with non-trivial kernel.
Let $\gamma \in \ker \psi$.
Since $G$ is residually nilpotent, there exists $i$ such that $g \notin \gamma_i(G)$.
Hence, the induced map $\rho_i: G/\gamma_i(G) \to \Gamma_g/\gamma_i(\Gamma_g)$ is onto but not bijective, which is impossible as finitely generated nilpotent groups are Hopfian. \qed

\paragraph{Proof of Theorem \ref{weaklyparasurfaceTheoremExamples}.} 
Let $G$, $w$, and $\gamma$ be as in the statement of Theorem \ref{weaklyparasurfaceTheoremExamples}.
We first show that $G$ is weakly parasurface:

\begin{claim} We have $w = 1$ in all quotients $G/\gamma_i(G)$. \label{allquotientsClaim}
\end{claim}

\emph{Proof of Claim.}
Let $H_1 = \left<\left< w[w,\gamma w \gamma^{-1} ] \right> \right>$ and $H_2 = \left<\left< w\right> \right>$ in $F_k/\gamma_i(F_k)$.
Clearly $H_1 \leq H_2$.
Further, the image of $H_1$ in $H_2/[H_2, H_2]$ generates as $w[w, \gamma w \gamma^{-1} ] = w \mod [H_2, H_2]$.
Hence, as $H_2$ is  nilpotent, Lemma \ref{generatorlemma} implies that $H_1 = H_2$, and so the claim follows.

\vskip.1in
If $w \neq 1$ in $G$, then the claim also shows that $G$ is not residually nilpotent.
Suppose, for the sake of a contradiction, that $w = 1$ in $G$.
Then by Magnus' conjugacy theorem for groups with one defining relator, $w^\epsilon$ and $w[w,\gamma w \gamma^{-1}]$ would be conjugate for $\epsilon = 1$ or $-1$, but this is impossible as they are both cyclically reduced words of different word lengths (Theorem 1.3, page 36, in \cite{MKS}).
Hence, $G$ is not residually nilpotent, and hence $G$ is not parasurface.
Further, as $w = 1$ in all nilpotent quotients, $G$ is weakly parasurface. 
The proof of Theorem \ref{weaklyparasurfaceTheoremExamples} is complete. \qed

\paragraph{Proof of Theorem \ref{parasurfaceTheoremExamples}.}
Let $G$ and $\delta$ be as in the statement of Theorem \ref{parasurfaceTheoremExamples}.
That $G$ is residually nilpotent follows from Theorem \ref{azarovTheorem} applied to $A = F_{k-2} = \left<a_1, a_2, \ldots, a_{k-2} \right>$ and $B = F_{2} = \left<a_{k-1}, a_k \right>$ with $\alpha = [a_1 \delta, a_2][a_3, a_4] \cdots [a_{k-3}, a_{k-2}]$ and $\beta = ([a_{k-1}, a_k])^{-1}$ and the following claim:
\begin{claim}
If $\beta = y^n$ in $B$, where $y \in B$, then $n = 1$.
\end{claim}

\emph{Proof of Claim.} If $\beta = y^n$ for some $n > 1$ then since $B/\left<\left< \beta \right>\right>$ is torsion--free, $y \in \left< \left< \beta \right> \right>$.
Hence, $\left< \left< y \right> \right> = \left< \left< \beta \right> \right>$, and so by Magnus' conjugacy theorem for groups with one defining relator, we have that $y$ is conjugate to $\beta $ or $\beta ^{-1}$.
But then $y^n$ would have cyclically reduced form $\beta ^n$, giving that $y^n$ has larger cyclically reduced word length than $\beta$.
This difference in word length cannot happen by Theorem 1.3, page 36, in \cite{MKS}, proving the claim.
\vskip.1in
We now show that $G$ is not a surface group.
Let $H = \left< a_1\delta, a_2, \ldots, a_{k-1}, a_k \right> \leq G$.
The next two claims imply that $H$ is a non-free proper subgroup of $G$ of rank $k$.
However, if $G$ were a surface group, it would have to be a surface group, $\Gamma_{k/2}$, of genus $k/2$.
Any rank $k$ proper subgroup of $\Gamma_{k/2}$ must be free, so $H$ must be free, a contradiction.

\begin{claim}$H$ is not a free group and is of rank $k$.\end{claim}

\emph{Proof of Claim.}
$H$ is rank $k$, because by Lemma \ref{generatorlemma}, $\{ a_1 \delta, a_2 , \ldots, a_k \}$ generate $G/[G,G]$ and $G/[G,G] = \mathbb{Z}^k$.
Further, if $H$ were free, it would be free of rank $k$.
Let $x_1, x_2, \ldots, x_k$ be a free basis for $H$.
The map $H \to H$ given by $x_1 \mapsto a_1 \delta$ and $x_k \mapsto a_k$ for $k > 1$
is an isomorphism because $H$ is Hopfian (being a free group).
Hence, H is freely generated by $\{ a_1 \delta, a_2, \ldots, a_k \}$, but this is impossible since $[a_1 \delta,a_2] \cdots [a_{k-1},a_k] = 1$.

\begin{claim}$H$ is a proper subgroup of $G$.\end{claim}

\emph{Proof of Claim.} Indeed, the element $a_1$ cannot be in $H$. For suppose $a_1 \in H$, and let $N$ be the normal subgroup generated by $a_k$ for $k > 2$.
Then $G/N = \left< a_1, a_2 : [a_1 \delta, a_2] \right>$. Since $a_1 \in H$, we have that $G/N$ is abelian, and so has presentation $G/N = \left< a_1, a_2: [a_1,a_2] \right>$.
But then the normal subgroup generated by $[a_1,a_2]$ and the normal subgroup generated by $[a_1\delta, a_2]$ are equal in $F_2 = \left<a_1, a_2 \right>$.
By Magnus' conjugacy theorem for groups with one defining relator, $[a_1,a_2]^\epsilon$ must be conjugate to $[a_1 \delta, a_2]$ for $\epsilon = 1$ or $-1$ in $F_2 = \left< a_1, a_2 \right>$, which is impossible as $[a_1,a_2]$ and $[a_1\delta,a_2]$ are cyclically reduced and have different word lengths (Theorem 1.3, page 36, in \cite{MKS}).
Hence $a_1 \notin H$, so the claim is proved.

We finish the proof of Theorem \ref{parasurfaceTheoremExamples} by showing that all of the lower central series quotients of $G$ match those of a surface group of genus $k/2$.
\begin{claim} \label{weakly2-parasurfaceclaim} $G$ is weakly $k/2$--parasurface.\end{claim}

\emph{Proof of Claim.} Let $\psi$ be the map defined by $a_1 \mapsto a_1 \delta$ and $a_k \mapsto a_k$ for $k > 1$.
This gives a well defined map $F_{k,i} \to F_{k,i}$
where $F_{k, i} := F_{k}/\gamma_i(F_{k})$.
This is an epimorphism by Lemma 1.
Since finitely generated nilpotent groups are Hopfian, $\psi : F_{k,i} \to F_{k,i}$ must be an isomorphism.
Therefore the induced map on $\Gamma_{k/2}/\gamma_i(\Gamma_{k/2}) \to G/\gamma_i(G)$ is an isomorphism, as claimed. The proof of Theorem \ref{parasurfaceTheoremExamples} is now complete. \qed

\smallskip \smallskip

We are now ready to quickly prove all of our theorems stated in the introduction.

\begin{proof}[Proof of Theorem \ref{magnusparasurfacethm}]
Theorem \ref{parasurfaceTheoremExamples} gives the desired parasurface groups.
\end{proof}

\begin{proof}[Proof of Theorem \ref{AutFreeNilpotentSeparableTheorem} ]
The proof of Claim \ref{weakly2-parasurfaceclaim} with $\delta = [[a_1,a_2],a_1]$ shows that $[a_1,a_2]\cdots [a_{k-1}, a_k]$ and $[a_1\delta, a_2]\cdots [a_{k-1}, a_k]$ are nilpotent--automorphism--equivalent.
However, $[a_1,a_2]\cdots [a_{k-1}, a_k]$ is not automorphism--equivalent to $[a_1\delta, a_2]\cdots [a_{k-1}, a_k]$ in $F_k$, by Theorem \ref{parasurfaceTheoremExamples}.
\end{proof}

\begin{proof}[Proof of Theorem \ref{conjugacytheoremNoNo}]
The element $w$ is not conjugate to $w[w,bwb^{-1}]$ in $F_4$.
Hence, as $F_4$ is conjugacy--nilpotent--separable, there exists some large $N > 0$ such that for all $i > N$ we have that $w$ is not conjugate to $w[w,bwb^{-1}]$ in $F_{4,i}$.
Moreover, the equality
$$\left< \left< w[w,bwb^{-1}] \right> \right>_{F_{4,i}}= \left<\left< w \right>\right>_{F_{4,i}},$$
for all $i$, is an immediate consequence of Claim \ref{allquotientsClaim}.
\end{proof}

\section{Final remarks}

In this article, we have shown that there exist groups which are almost surface groups in the sense that they share all their lower central series quotients with a surface group but are not themselves surface groups.
In light of our examples, we pose the following question.

\begin{question}
What properties to parasurface groups share with surface groups?
\end{question}

As a small step in answering this question, we present the following

\begin{theorem} \label{positivetheorem}
Any finite-index subgroup of a parasurface group is not free.
\end{theorem}

\begin{proof}
Let $G$ be a parasurface group.
Note that $G$ is not free, for if it were, Magnus' theorem would imply that the fundamental group of some compact surface is free.
Further, $G$ is torsion--free, for otherwise by residual nilpotence, there would exist torsion elements in $\Gamma_g/\gamma_k(\Gamma_g)$ for some $g$ and $k$, which is impossible by Labute \cite{labute}.

Let $\cd(G)$ denote the cohomological dimension of $G$.
If $\Gamma \leq G$ is a free group of finite index, then by Theorem 3.1 on page 190 in \cite{brown} and the fact that $G$ is torsion--free, $\cd G = \cd \Gamma$.
Hence, $\cd G = 1$, but then by Stallings \cite{stallings} and Swan \cite{swan}, $G$ must itself be free, a contradiction.
\end{proof}


\noindent Department of Mathematics \\
University of Chicago \\
Chicago, IL 60637, USA \\
email: {\tt khalid@math.uchicago.edu}
\end{document}